\documentclass[11pt]{amsart}
\usepackage[right=3cm, left=3cm, top=2cm]{geometry}
\newtheorem{theorem}{Theorem}
\newtheorem{prop}[theorem]{Proposition}
\newtheorem{cor}[theorem]{Corollary}
\newtheorem{lemma}[theorem]{Lemma}

\theoremstyle{definition}
\newtheorem{definition}[theorem]{Definition}

\newcommand{\cc}{\mathbb{C}}

\renewcommand{\epsilon}{\varepsilon}

\DeclareMathOperator{\re}{Re}

\newcommand{\charac}{\raisebox{0.4ex}{$\chi$}}

\newcommand{\Rmnum}[1]{\expandafter\@\romancap\romannumeral #1@}

\newcommand{\fh}{\mathfrak{H}}
\newcommand{\fk}{\mathfrak{K}}
\newcommand{\fb}{\mathfrak{B}}

\newcommand{\bh}{\fb(\fh)}

\newcommand{\fs}{\mathfrak{S}}
\newcommand{\subtriv}{\fs}

\newcommand{\id}{\mathrm{id}}

\begin{document}

\title{Aligned CP-semigroups}
\author{Christopher Jankowski}
\address{Christopher Jankowski, Department of Mathematics,
Ben-Gurion University of the Negev,
P.O.B. 653, Be'er Sheva 84105,
Israel.}
\email{cjankows@math.bgu.ac.il}
\thanks{C.J. was partially supported by the
Skirball Foundation via
the Center for Advanced Studies in Mathematics at Ben-Gurion
University of the Negev.
}

\author{Daniel Markiewicz}
\address{Daniel Markiewicz, Department of Mathematics,
Ben-Gurion University of the Negev,
P.O.B. 653, Be'er Sheva 84105,
Israel.}
\email{danielm@math.bgu.ac.il}

\thanks{D.M. and R.T.P. were partially supported by
grant 2008295 from the U.S.-Israel Binational Science Foundation.
}

\author{Robert T. Powers}
\address{Robert T. Powers, Department of Mathematics,
David Rittenhouse Lab.,
209 South 33rd St.,
Philadelphia, PA 19104-6395, U.S.A.}
\email{rpowers@math.upenn.edu}

\begin{abstract}
A CP-semigroup is aligned if its set of trivially maximal subordinates is totally ordered by subordination.  We prove that aligned spatial E$_0$-semigroups are prime: they have no non-trivial tensor product decompositions up to cocycle conjugacy. As a consequence, we establish the existence of uncountably many non-cocycle conjugate E$_0$-semigroups of type II$_0$ which are prime.
\end{abstract}
\maketitle

Let $\fh$ be a Hilbert space, which we will always assume to be separable and infinite-dimensional, and let $\bh$ denote the $*$-algebra
of all bounded operators over $\fh$.
A \emph{CP-semigroup acting on $\bh$} is a point-$\sigma$-weakly continuous semigroup $\alpha =\{\alpha_t: \bh \to \bh\}_{t\geq0}$ of normal completely positive contractions
such that $\alpha_0 = \id$. When $\alpha_t$ is an endomorphism and $\alpha_t(I)=I$ for all $t\geq 0$, then $\alpha$ is called an \emph{E$_0$-semigroup}.
In the special case when $\fh = \fk \otimes L^2(0,\infty)$, a CP-semigroup $\alpha$ acting on $\bh$ is called a \emph{CP-flow over $\fk$} if $\alpha_t(A)S_t = S_tA$ for all $A \in \bh$, $t\geq 0$ where $\{ S_t : t\geq 0\}$ is the right shift semigroup. We will say that a CP-semigroup $\beta$ is a \emph{subordinate} of $\alpha$ if it also acts on $\bh$ and $\alpha_t -\beta_t$ is completely positive. If in addition $\beta$ is a CP-flow, then it is called a \emph{flow subordinate} of $\alpha$. We direct the reader to \cite{arv-monograph} for a general reference on the theory of CP-semigroups and to \cite{powers-CPflows} for the basic theory of CP-flows.

In this paper we study a class of CP-semigroups which has a set of subordinates which is minimal in a natural sense. We call such CP-semigroups aligned. This class is shown to include examples considered previously in \cite{powers-holyoke, alevras-powers-price, jankowski1, jmp1}. We prove that aligned E$_0$-semigroups have a notable property: they are prime in the sense that they have no non-trivial tensor product decompositions up to cocycle conjugacy. As a consequence, we establish the existence of uncountably many non-cocycle conjugate E$_0$-semigroups of type II$_0$ which are prime. The previously known non-trivial examples of prime E$_0$-semigroups were obtained independently in \cite{markiewicz-powers} (type II$_1$), \cite{tsirelson-transitivity} (type II$_1$) and \cite{liebscher} (type II$_k$ for $k=1,2,\dots$).

Powers introduced in ~\cite{powers-holyoke} the concept of $q$-purity which has been valuable for the study of E$_0$-semigroups (see for example \cite{alevras-powers-price, jankowski1}). The notion of $q$-purity was refined in \cite{jmp1}: a CP-flow is \emph{$q$-pure} if and only if its set of flow subordinates is totally ordered by subordination. It is clear that an E$_0$-semigroup which is in addition a $q$-pure CP-flow must have index 0 or 1, and must be of type I$_1$, II$_0$ or II$_1$. The case of type I$_0$ is excluded because an automorphism group cannot be a CP-flow.

The restriction to flow subordinates in the definition of $q$-pure CP-flows can obscure some useful properties of these CP-semigroups. In order to circumvent this difficulty we consider an alternative concept which is inspired by the notion of $q$-purity.

\begin{definition}\label{q-pure}
Let $\alpha$ be a CP-semigroup and let $\beta$ be a CP-semigroup subordinate of $\alpha$. We will say that $\beta$ is \emph{trivially maximal} if the semigroup $e^{kt}\beta_t$ is not a subordinate of $\alpha$ for $k>0$. We will denote by
$\subtriv(\alpha)$ the set of all trivially maximal subordinates of $\alpha$ partially ordered by subordination. We will say $\alpha$ is \emph{aligned} if $\subtriv(\alpha)$ is totally ordered.

Let $\beta$ be a trivially maximal subordinate of $\alpha$. We will denote by $\subtriv(\alpha; \beta)$ the set of all trivially maximal CP-semigroup subordinates of $\alpha$ which dominate $\beta$, partially ordered by subordination. We will say that $\alpha$ is \emph{aligned relative to $\beta$} if $\subtriv(\alpha; \beta)$ is totally ordered.
\end{definition}

Notice that if an aligned E$_0$-semigroup is spatial, then it must have index zero.

\begin{lemma}\label{lemma}
 A unital CP-semigroup is aligned if and only if its minimal dilation is aligned.
\end{lemma}
\begin{proof}
 Suppose that $\alpha$ is a unital CP-semigroup acting on $\fh$ with minimal dilation $\alpha^d$ acting on $\fb(\fh_1)$, i.e. there exists an isometry $W: \fh \to \fh_1$ such that
$$
\alpha_t(x)= W^*\alpha_t^d(WxW^*)W
$$
and $WW^*$ is increasing for $\alpha^d$. In order to prove the statement it suffices to show that there is an order isomorphism between $\subtriv(\alpha)$ and $\subtriv(\alpha^d)$. As proved in \cite{bhat-cocycles} or by Theorem~3.5 of \cite{powers-CPflows}, there exists an order isomorphism between the set of CP-semigroup subordinates of $\alpha$ and the set of CP-semigroup subordinates of $\alpha^d$. This isomorphism is described as follows. For every subordinate $\beta$ of $\alpha$ there exists a unique subordinate $\beta'$ of $\alpha^d$ such that
\begin{equation}\label{the-beta}
\beta_t(x) = W^*\beta_t'(WxW^*)W.
\end{equation}
It is clear that if $\beta$ is trivially maximal (with respect to $\alpha$), then $\beta'$ is trivially maximal (with respect to $\alpha^d$). Conversely, suppose that $\beta$ is not  trivially maximal, so that there exists $k>0$ such that $e^{kt}\beta_t$ is a subordinate of $\alpha$. Then there exists $\gamma$ a unique subordinate of $\alpha^d$ such that
$$
e^{kt} \beta_t(x) = W^*\gamma_t(WxW^*)W, \forall x \in \bh, t\geq 0.
$$
Therefore, by dividing by $e^{kt}$ we obtain that $\beta$ is also the compression of $e^{-kt}\gamma_t$ which is obviously also a subordinate since $k>0$. By uniqueness of the correspondence, $\beta'_t = e^{-kt}\gamma_t$ for all $t$. Hence $\beta'$ is not trivially maximal.
\end{proof}

CP-flow subordinates of a CP-flow are always trivially maximal,
 therefore if a CP-flow is aligned then it is automatically $q$-pure. We also note that a CP-flow is $q$-pure if and only if it is aligned with respect to the subordinate $x \mapsto S_txS_t^*$. We omit the elementary proof of this fact.

We now show that unital CP-flows are aligned if and only if they are $q$-pure and induce E$_0$-semigroups of type II$_0$. First let us approach the case when the CP-flow is in also a semigroup of endomorphisms.
\begin{prop}
Let $\alpha$ be an E$_0$-semigroup which is in addition a CP-flow over $\fk$. If $\alpha$ has type II$_0$, then $\mathfrak{S}(\alpha)$ is precisely the set of flow subordinates of $\alpha$. Therefore, $\alpha$ is aligned if and only if it is type II$_0$ and it is $q$-pure.
\end{prop}
\begin{proof}
Suppose that $\alpha$ is an E$_0$-semigroup of type II$_0$ which is also a CP-flow over $\fk$.
Note that every flow subordinate of $\alpha$ is clearly an element of $\mathfrak{S}(\alpha)$, as flow subordinates are always trivially maximal. Conversely, let
 $\beta$ be a trivially maximal CP-semigroup subordinate of $\alpha$. By Theorem~3.4 of \cite{powers-CPflows} there exists a family of operators $(C(t))_{t\geq0}$ in $\bh$ such that
$$
\beta_t(x) = C(t)\alpha_t(x),
$$
where the family $(C(t))_{t \geq 0}$ is a contractive positive local cocycle, i.e. $C(t) \in \alpha_t(\bh)'$ and $0\leq C(t) \leq I$
for all $t>0$, $C(t+s) = C(t)\alpha_t(C(s))$ for all $t,s \geq 0$ and $t \mapsto C(t)$ is strongly continuous for $t\geq 0$
with $C(t) \to I$ as $t\to 0+$.

Let $S_t$ denote as usual the right shift semigroup on $\fh = \fk \otimes L^2(0,\infty)$ and let $V_t = C(t) S_t$. Note
that $V_t$ is strongly continuous and furthermore it is a semigroup: for all $t,r \geq 0$,
$$
V_t V_r = C(t) S_t C(r) S_r = C(t) \alpha_t(C(r)) S_t S_r = C(t+r) S_{t+r} = V_{t+r}.
$$
Moreover, $V_t$ is a unit of $\alpha_t$, since for all $x \in \bh$,
$$
\alpha_t(x) V_t = \alpha_t(x) C(t) S_t = C(t) \alpha_t(x) S_t = C(t) S_tx = V_t x.
$$
Notice, however, that $\alpha$ is type II$_0$, therefore there exists $\kappa \in \cc$ such that
$V_t = e^{-\kappa t} S_t$ for all $t\geq 0$. Furthermore, since $C(t)$ is a contraction, we have
that $V_t$ is a contraction for all $t>0$, hence we must have $\re(\kappa) \geq 0$. We now show
that in fact $\kappa$ must be a real number. Recall that $\beta_t$ is a CP-semigroup, and
$$
0 \leq S_t^*\beta_t(I) S_t = S_t^* C(t) S_t = e^{-\kappa t} I,
$$
hence $\kappa$ is real and satisfies $\kappa \geq 0$.

We will now prove that $\kappa=0$. Let $\gamma_t(x) = e^{\kappa t} \beta_t(x)$. Notice that $S_t$
is an intertwiner semigroup for $\gamma$: for all $t>0$ and $x\in\bh$,
$$
\gamma_t(x) S_t = e^{\kappa t} \beta_t(x) S_t = e^{\kappa t} C(t) \alpha_t(x) S_t
 = e^{\kappa t} C(t) S_t x = S_t x.
$$
It follows that $\gamma$ is a CP$_\kappa$-flow in the sense of Definition 4.0 of
 \cite{powers-CPflows}, that is to say, $e^{-\kappa t} \gamma_t$ is a CP-semigroup
 and $\gamma$ is intertwined by the shift. However, by Theorem~4.15 of
\cite{powers-CPflows}, every CP$_\kappa$-flow must be a CP-flow. But a CP-flow
is contractive, hence we must have $\gamma_t(I) \leq I$, thus
$$
e^{\kappa t} \beta_t(I) = e^{\kappa t} C(t) \leq I
$$
for all $t>0$. Therefore, we have that for all positive $t>0$,
$$
e^{kt}\beta_t(x) = e^{\kappa t} C(t) \alpha_t(x) \leq \alpha_t(x).
$$
Since $\beta$ is trivially maximal, we must have that $\kappa =0$. Thus we have shown that every element of $\mathfrak{S}(\alpha)$ is a CP-flow.

Therefore, if $\alpha$ is $q$-pure and of type II$_0$ then it is aligned. On the other hand, it is clear that if $\alpha$ is an aligned E$_0$-semigroup, then it is type II$_0$ and $q$-pure as discussed before the proposition.
\end{proof}

\begin{theorem}\label{aligned-and-q-pure}
Let $\alpha$ be a unital CP-flow over $\fk$. If the minimal dilation of $\alpha$ is type II$_0$, then $\mathfrak{S}(\alpha)$ is precisely the set of flow subordinates of $\alpha$. Therefore, $\alpha$ is aligned if and only if it is $q$-pure and its minimal dilation is type II$_0$.
\end{theorem}
\begin{proof}
By Lemma~4.50 of \cite{powers-CPflows}, there exists a minimal dilation of $\alpha$ to an E$_0$-semigroup $\alpha^d$ which is a CP-flow over the Hilbert space $\fh_1 = \fk_1 \otimes L^2(0,\infty)$, i.e. there exists an isometry $W: \fh \to \fh_1$ such that
$$
\alpha_t(x)= W^*\alpha_t^d(WxW^*)W,
$$
$WW^*$ is increasing for $\alpha^d$ and if $S_t^d$ denotes the right shift semigroup on the space $\fh_1$, then $WS_t = S_t^dW$ for all $t>0$. We use the order isomorphism established in the proof of Lemma~\ref{lemma} that associates to each subordinate $\beta \in \mathfrak{S}(\alpha)$  a unique subordinate $\beta'\in \mathfrak{S}(\alpha^d)$ satisfying \eqref{the-beta}.
If $\alpha^d$ is type II$_0$ and $\beta \in \mathfrak{S}(\alpha)$, then it follows from the previous proposition that $\beta'$ is a CP-flow. Hence for all $t>0$,
$$
\beta_t(x)S_t =  W^*\beta_t'(WxW^*)WS_t =W^*\beta_t'(WxW^*)S_t^dW = W^*S_t^dWxW^*W = S_tx.
$$
In other words, $\beta$ is a CP-flow. Thus we proved that all elements of $\mathfrak{S}(\alpha)$ are CP-flows. On the other hand, every flow subordinate of $\alpha$ is clearly an element of $\mathfrak{S}(\alpha)$.

Thus, if $\alpha$ is a unital $q$-pure CP-flow and its minimal dilation has type II$_0$, then it is aligned. Conversely, it is clear that if $\alpha$ is aligned then its minimal dilation $\alpha^d$ as discussed above is also aligned, thus it has index zero. Since $\alpha^d$ is a CP-flow, it cannot be an automorphism group, hence it has type II$_0$. Moreover, since $\alpha$ is aligned, it is clearly $q$-pure.
\end{proof}

\section*{Prime E$_0$-semigroups}

\begin{definition}
An E$_0$-semigroup $\alpha$ is called \emph{prime} if, whenever $\alpha$ is cocycle conjugate to $\beta \otimes \gamma$ where $\beta$ and $\gamma$ are E$_0$-semigroups, then $\beta$ or $\gamma$ is type I$_0$.
\end{definition}

It follows from the complete classification of E$_0$-semigroups of type I in terms of the
index, and the additivity of the index with respect to tensoring, that a prime E$_0$-semigroup
of type I is cocycle conjugate either to an automorphism group or to the CAR/CCR flow of
index 1. It is a corollary of the work on the gauge group of an E$_0$-semigroup by Markiewicz-Powers in \cite{markiewicz-powers} or Tsirelson in \cite{tsirelson-transitivity}, that prime E$_0$-semigroups of type II$_1$ exist. And,
as it has belatedly come to our attention,  earlier Liebscher had proven that prime
E$_0$-semigroups of type II$_k$ exist  for $k\geq 1$ (see Proposition 4.32 and Note~4.33 in
\cite{liebscher}).
We establish that there exist uncountably many prime E$_0$-semigroups of type II$_0$.

\begin{theorem}\label{aligned-is-prime}
Aligned spatial E$_0$-semigroups are prime.
\end{theorem}
\begin{proof}
We prove the contrapositive. Suppose that $\alpha$ is an E$_0$-semigroup and $\alpha = \beta
 \otimes \gamma$ where $\beta$ and $\gamma$ are two spatial E$_0$-semigroups neither of which
has type I$_0$. Without loss of generality, by applying an appropriate conjugacy, we
may assume that both act on $\bh$ where $\fh$ is infinite-dimensional and separable. Let
$U$ and $V$ be normalized units of $\beta$  and $\gamma$, respectively. Let us denote
by $\theta^U$ and $\theta^V$ the semigroups given by $\theta^U_t(x) = U_txU_t^*$ and
 $\theta^V_t(x)=V_txV_t^*$, which are E-semigroup subordinates of $\beta$ and $\gamma$,
 respectively (notice these are not unital since $\beta$ and $\gamma$ are not automorphism groups). Notice that $\sigma^V=\beta \otimes \theta^V$ is a subordinate of $\alpha$.
We show that it is trivially maximal. Suppose that $k \geq 0$ and $e^{kt}\sigma^V_t$ is also a
 subordinate of $\alpha$. Then we have that for all $x \in \bh$,
$$
 e^{kt} (I \otimes V_tV_t^*) = e^{kt} (\beta_t \otimes \theta^V_t)(I) = e^{kt} \sigma^V_t(I)  \leq \alpha_t(I) = I
$$
However $\| V_t \|=1$, hence by taking norms on both sides we conclude $e^{kt} \leq 1$ for all $t>0$. Thus $k=0$. Analogously, $\sigma^U=\theta^U \otimes \gamma$ is trivially maximal.

We prove that  $\sigma^V$ and $\sigma^U$ are incomparable. Suppose that $\sigma^V \geq \sigma^U$. Notice that for all $x,y \in \bh$ and $t>0$,
\begin{align*}
\sigma^V_t(x \otimes y) (U_t \otimes I) & =
\alpha_t(x)U_t \otimes \theta_t^V(x) = (U_t \otimes I) (x \otimes \theta^V_t(y)) \\
\sigma^U_t(x \otimes y) (U_t \otimes I) & =
\theta^U_t(x)U_t \otimes \gamma_t(x) = (U_t \otimes I) (x \otimes \gamma_t(y))
\end{align*}
Therefore we have that for all $x, y \in \bh$ and $t>0$,
$$
(U_t \otimes I)^* \Big[\sigma^V_t(x \otimes y) - \sigma^U_t(x \otimes y) \Big](U_t \otimes I) =
x \otimes [ \theta_t^V(y) - \gamma_t(y) ]
$$
Thus in the special case when $x=I$, we have that the map $y \mapsto I\otimes [\theta_t^V(y) - \gamma_t(y)]$ is completely positive for all $t>0$, hence $\theta^V \geq \gamma$. However $\theta^V$ is a pure element in the cone of completely positive maps, therefore we have that for every $t>0$, $\gamma_t$ is a multiple of $\theta_t^V$, and we have a contradiction since $\gamma$ is not type I$_0$. By symmetry, we obtain that $\sigma^U$ and $\sigma^V$ are incomparable as asserted. Thus we proved that $\alpha$ is not aligned.
\end{proof}

\begin{cor}\label{many-primes}
There exist uncountably many non-cocycle conjugate E$_0$-semigroups of type II$_0$ which are prime.
\end{cor}
\begin{proof}
By Theorem~\ref{aligned-is-prime}, it suffices to exhibit an uncountable family of non-cocycle conjugate aligned E$_0$-semigroups.  In order to do so, we consider a class of E$_0$-semigroups arising from boundary weight doubles as in \cite{jankowski1}. Let $g(x)$ be a fixed complex-valued measurable function such that $g \not \in L^2(0,\infty)$ yet $(1-e^{-x})^{1/2} g(x) \in L^2(0,\infty)$, and for each $t>0$ let $g_t=\charac_{(t,\infty)}g$, which is an element of $L^2(0,\infty)$. Define the weight on $\fb(L^2(0,\infty))$ given by
$\nu(A) = \lim_{t\to 0+} (g_t,Ag_t)$. For every $0 < \lambda <1/2$, let $\mu_\lambda: M_2(\cc) \to \cc$ be given by
$$
\mu_\lambda(X) = \lambda x_{11} + (1-\lambda) x_{22}.
$$
Let us define the boundary weight map from $M_2(\cc)_*$ to boundary weights on $\fb(\cc^2 \otimes
L^2(0,\infty))$ given by $\omega(\rho)(A) =
\rho(I) \mu_\lambda\big((I \otimes \nu)(A)\big)$ for all $\rho \in M_2(\cc)$ and $A$ in the
domain of finiteness of $I \otimes \nu$ (the so called null boundary algebra of definition 4.16
in \cite{powers-CPflows}). Then by Corollary~3.3 of \cite{jankowski1} and the assumptions on
$g(x)$, we have that $\omega$ gives rise to an E$_0$-semigroup of type II$_0$. Once one applies
Theorem~4.4 of \cite{jmp1} to reconcile the earlier definition of $q$-purity with the
one in this paper, we obtain from Lemma~4.3 and Proposition~5.2 of \cite{jankowski1} that $\omega$ gives
rise to a $q$-pure unital CP-flow. Therefore by Theorem~\ref{aligned-and-q-pure} it gives to a
aligned E$_0$-semigroup $\alpha^{\lambda}$. Finally, by Theorem~5.4 of \cite{jankowski1},  given
$\lambda, \zeta \in (0,1/2)$, we have that $\alpha^{\lambda}$ is cocycle conjugate to
$\alpha^\zeta$ if and only if $\lambda =\zeta$.
\end{proof}

We should point out that it is possible to obtain a different uncountable family of non-cocycle conjugate E$_0$-semigroups by using Theorem~3.22 of \cite{powers-holyoke}. For details see the discussion in the end of section III therein. This would be indeed be a different family from the one exhibited above in the sense that, by Corollary 5.5 of \cite{jankowski1}, the E$_0$-semigroups constituting both families are not cocycle conjugate.

\providecommand{\bysame}{\leavevmode\hbox to3em{\hrulefill}\thinspace}


\begin{thebibliography}{Pow03b}

\bibitem[APP06]{alevras-powers-price}
A.~Alevras, R.~T. Powers, and G.~L. Price, \emph{Cocycles for one-para\-meter
  flows of {$B(H)$}}, J. Funct. Anal. \textbf{230} (2006), no.~1, 1--64.

\bibitem[Arv03]{arv-monograph}
W.~Arveson, \emph{Noncommutative dynamics and {$E$}-semigroups}, Springer
  Monographs in Mathematics, Sprin\-ger-Ve\-rlag, New York, 2003.

\bibitem[Bha01]{bhat-cocycles}
B.~V.~R. Bhat, \emph{Cocycles of {CCR} flows}, Mem. Amer. Math. Soc.
  \textbf{149} (2001), no.~709, x+114.

\bibitem[Jan10]{jankowski1}
C.~Jankowski, \emph{On type {$\rm II_0$} {$E_0$}-semigroups induced by boundary
  weight doubles}, J. Funct. Anal. \textbf{258} (2010), no.~10, 3413--3451.

\bibitem[JMP11]{jmp1}
C.~Jankowski, D.~Markiewicz, and R.~T. Powers, \emph{E$_0$-semigroups and
  $q$-purity}, preprint arxiv:1106.2304, 2011.

\bibitem[Lie09]{liebscher}
V.~Liebscher, \emph{Random sets and invariants for (type {II}) continuous
  tensor product systems of {H}ilbert spaces}, Mem. Amer. Math. Soc.
  \textbf{199} (2009), no.~930, xiv+101.

\bibitem[MP09]{markiewicz-powers}
D.~Markiewicz and R.~T. Powers, \emph{Local unitary cocycles of spatial
  ${E}_0$-semigroups}, J. Funct. Anal. \textbf{256} (2009), no.~5, 1511--1543.

\bibitem[Pow03a]{powers-holyoke}
R.~T. Powers, \emph{Construction of {$E\sb 0$}-semigroups of {${B}({H})$} from
  {CP}-flows}, Advances in quantum dynamics ({S}outh {H}adley, {MA}, 2002),
  Contemp. Math., vol. 335, Amer. Math. Soc., Providence, RI, 2003, pp.~57--97.

\bibitem[Pow03b]{powers-CPflows}
\bysame, \emph{Continuous spatial semigroups of completely positive maps of
  {$B(H)$}}, New York J. Math. \textbf{9} (2003), 165--269 (electronic).

\bibitem[Tsi08]{tsirelson-transitivity}
B.~Tsirelson, \emph{On automorphisms of type {II} {A}rveson systems
  (probabilistic approach)}, New York J. Math. \textbf{14} (2008), 539--576.

\end{thebibliography}
\end{document}